\documentclass[12pt]{amsart}

\usepackage{amsmath,amssymb,amsfonts,verbatim}

\newtheorem{theorem}{Theorem}[section]
\newtheorem{lemma}[theorem]{Lemma}

\begin{document}

\title[non-nilpotent permutable ]{Locally graded groups with all non-nilpotent subgroups permutable, II}

\author{Sevgi Atlihan}
\address[Sevgi Atlihan]
{Department of Mathematics\\
Faculty of Education\\
Gazi University\\
Technikokkular, Ankara, Turkey}
\email{asevgi@gazi.edu.tr}
\author{Martyn R. Dixon}
\address[Martyn Dixon]
{Department of Mathematics\\
University of Alabama\\
Tuscaloosa, AL 35487-0350, U.S.A.}
\email{mdixon@ua.edu}
\author{Martin J. Evans}
\address[Martin Evans]
{Department of Mathematics\\
University of Alabama\\
Tuscaloosa, AL 35487-0350, U.S.A.}
\email{mevans@ua.edu}

\begin{abstract}  Let $G$ be a locally graded group and suppose that every non-nilpotent subgroup of $G$ is permutable. We prove that $G$ is soluble. (In light of previous results of the authors, it suffices to prove that $G$ is soluble if it is periodic.) 
\end{abstract}

\dedicatory{ Dedicated to Derek Robinson on his 85th birthday} 



\maketitle

\newcommand{\sub}[2]{\langle#1,#2\rangle}
\newcommand{\cyclic}[1]{\langle #1\rangle}
\newcommand{\norm}{\triangleleft\,}
\newcommand{\core}[2]{\text{core}\,_{#1}\,#2}
\newcommand{\cpin}{C_{p^{\infty}}}
\newcommand{\co}[1]{\text{$\mathbf{#1}$}}
\newcommand{\cogo}[2]{\co{#1}\mathfrak{#2}}
\newcommand{\ann}{\text{Ann\,}}
\newcommand{\spec}{\text{Spec\,}}
\newcommand{\listel}[2]{#1_1,\dots, #1_{#2}}
\newcommand{\aut}{\text{Aut}\,}
\newcommand{\cent}[2]{C_{#1}(#2)}
\newcommand{\Dr}{\text{Dr}\,}
\newcommand{\dir}[3]{\underset{#1\in #2}{\Dr}#3_{#1}}
\newcommand{\nocl}[2]{\langle #1\rangle^{#2}}
\newcommand{\qn}{\,\text{qn}\,}
\newcommand{\comm}[2]{#1^{-1}#2^{-1}#1#2}
\newcommand{\conj}[2]{#1^{-1}#2#1}
\newcommand{\izer}[2]{N_{#1}(#2)}
\newcommand{\gl}{\text{GL\,}}
\newcommand{\subgp}[3]{\langle#1_{#2}\mid #2\in#3\rangle}
\newcommand{\direct}[3]{\underset{#1\in #2}{\Dr}\langle#3_{#1}\rangle}
\newcommand{\omn}[1]{\Omega_n{(#1)}}
\newcommand{\Hom}{\text{Hom}\,}
\newcommand{\Cr}[3]{\underset{#1\in #2}{\text{Cr}}\,#3_{#1}}
\newcommand{\inflist}[1]{\{#1_1,#1_2,\dots\}} 
\newcommand{\cart}{\text{Cr}\,} 
\newcommand{\carti}[3]{\underset{#1\geq #2}\cart #3_{#1}} 
\newcommand{\fgsub}[2]{\langle #1_1,\dots,#1_{#2}\rangle}
\newcommand{\cwr}{\,\bar{\wr}\,}
\newcommand{\minn}{\infty$-$\overline{\mathfrak{N}}}
\newcommand{\mnn}{\overline{\mathfrak{N}}}
\newcommand{\mnp}{\overline{\mathcal{P}}}
\newcommand{\mns}{\overline{\mathfrak{S}}}
\newcommand{\minsd}{\infty$-$\overline{\mathfrak{S}_d}}
\newcommand{\minnc}{\infty$-$\overline{\mathfrak{N}_c}}
\newcommand{\ms}{\mathfrak{S}}
\newcommand{\mn}{\mathfrak{N}}
\newcommand{\mf}{\mathfrak{F}}
\newcommand{\sdr}{\mathfrak{S}_d(r)}
\newcommand{\sd}{\mathfrak{S}_d}
\newcommand{\mfr}{\mathfrak{R}}
\newcommand{\ncr}{\mathfrak{N}_c(r)}
\newcommand{\mfsr}{\overline{\ms\mfr}}
\newcommand{\lslf}{\overline{(\cogo{L}{S})(\cogo{L}{F})}}
\newcommand{\mcp}{\mathcal{P}}
\newcommand{\mcps}{\mathcal{P}^*}
\newcommand{\mfas}{\mathfrak{A}^*}
\newcommand{\mfa}{\mathfrak{A}}
\newcommand{\mfss}{\mathfrak{S}^*}
\newcommand{\mfns}{\mathfrak{N}^*}
\newcommand{\wmcnq}{max-$\infty$-$\overline{\text{qn}}$ }
\newcommand{\om}{\omega}
\newcommand{\bog}[1]{\bar{\omega}(#1)}
\newcommand{\og}[1]{\omega(#1)}
\newcommand{\wi}{\overline{\om}_i}
\newcommand{\w}{\overline{\om}}
\newcommand{\fsn}{\emph{f}-subnormal }
\newcommand{\wiw}{\om_i}
\newcommand{\ww}{\om}

\section{Introduction} \label{s:intro}
In a previous work \cite{ADE23} the authors considered locally graded groups $G$ in which every non-nilpotent subgroup is permutable. It was proved that: 
\begin{enumerate}
\item[(i)] If $G$ is torsion-free, then $G$ is nilpotent.
\item[(ii)] If $G$ is not periodic, then $G$ is soluble.
\end{enumerate}

In this note we complete the picture by establishing the following theorem.
\begin{theorem}\label{t:periodic}
 Let $G$ be a locally graded group and suppose that every non-nilpotent subgroup of $G$ is permutable. If $G$ is periodic, then $G$ is soluble.

\end{theorem}

Clearly, when combined with (ii) above, this theorem implies the following, which is our main result. 

\begin{theorem}\label{t:main}
Let $G$ be a locally graded group and suppose that every non-nilpotent subgroup of $G$ is permutable. Then $G$ is soluble.

\end{theorem}

For the sake of brevity, we direct the reader to \cite{ADE23} for definitions, background on permutable subgroups and quasihamiltonian groups, and motivation for considering groups of this type. However, we would be remiss if we did not point out that 
Theorem \ref{t:periodic} can be viewed as a `permutable' analogue of a result of Smith \cite{hS01b} which asserts that a locally finite group in which every non-nilpotent subgroup is \emph{subnormal} is soluble.

\section{Three Lemmas and the proof of Theorem \ref{t:periodic}}

\begin{lemma}\label{l:omni}
Let $G$ be a non-soluble periodic locally graded group and suppose every non-nilpotent subgroup of $G$ is permutable. Let $X=G''$. Then
\begin{enumerate}
\item[  (i)] $X$ is perfect; 
\item[(ii)] Every proper normal subgroup of $X$ is nilpotent;
\item[ (iii)] $X$ is a Fitting group;
\item[(iv)] Every proper subgroup of $X$ is soluble;
\item[(v)] $X$ is a $p$-group for some prime $p$.
\end{enumerate}
\end{lemma}

\begin{proof}
Parts (i)-(iv)  are established in \cite[Lemmas 2.1 and 2.2]{ADE23} and so it only remains to show that $X$ is $p$-group for some prime $p$. Suppose not. Since $X$ is a non-trivial  periodic Fitting group, it is the direct product of two normal proper subgroups, each of which is nilpotent by part (ii) of this proposition. This implies that $X$ is nilpotent,  contrary to part (i), and the desired result follows. .
\end{proof}

Our remaining lemmas record two pieces of folklore.

\begin{lemma}\label{l:folk1}
 Let $H$ be a subgroup of finite index in a group $L$ and supose that $H$ is residually finite. Then $L$ is residually finite.
\end{lemma}
\begin{proof}
Since $H$ is residually finite, there exists a collection of  normal subgroups $\{N_i\}_{i\in I}$ of $H$ such that each $N_i$ has finite index in $H$  and the intersection of all the $N_i$ is trivial. 
For each $i\in I$, let $M_i$ denote the core of $N_i$ in $L$. Then each $M_i$ is normal and of finite index in $L$. Moreover, $M_i\leq N_i$ for each $i\in I$, and so the intersection of all the $M_i$ is trivial. The result follows.
\end{proof}

\begin{lemma} \label{l:folk2} 
 Let $N$ be a nilpotent subgroup of a Fitting group $X$ and suppose that $N$ is both normal and of finite index in a subgroup $K$ of $X$. Then $K$ is nilpotent.
\end{lemma}

\begin{proof}
Let $x_1,x_2,\dots,x_n$ be a set of coset representatives for $N$ in $K$and note that $K=\langle N,x_1,x_2,\dots,x_n\rangle$. For each $i=1,2,\dots,n$, let $L_i$ denote the normal closure of $x_i$ in $X$ and $M_i$ the normal closure of $x_i$ in $K$. Since $X$ is a Fitting group, each $L_i$ is a normal \emph{nilpotent} subgroup of $X$ and so each $M_i$ is a nilpotent normal subgroup of $K$. It follows that $K=\langle N,M_1,M_2,\dots,M_n\rangle$ is the join of finitely many of its normal nilpotent subgroups and is therefore nilpotent.   
\end{proof}

\begin{proof}[\bf{The Proof of Theorem \ref{t:periodic}}]
Let $G$ be a periodic locally graded group in which every non-nilpotent subgroup is permutable, and suppose, for a contradiction, that $G$ is not soluble. Let $X=G''$ and note that $X$ has all the properties listed in Lemma \ref{l:omni}. Since $X$ is locally nilpotent and perfect, it  follows that $X$ is infinite. Moreover, if all proper subgroups of $X$ are nilpotent, a celebrated result of Asar \cite{aA99} implies that $X$ is soluble and this contradiction shows that $X$ has a proper non-nilpotent subgroup, $P$ say, which is necessarily a permutable subgroup of $X$. Let $P_X$ denote the core of $P$ in $X$.  Lemma \ref{l:omni} shows that $P_X$ is soluble and we let $d$ denote its derived length.  Clearly $P/P_X$ is a core-free permutable subgroup of $X/P_X$, and so a result of Stonehewer \cite[Theorem C]{sS72} now implies that $P/P_X$ is residually finite. Let $F$ be a finitely generated subgroup of $X$; of course, $F$ is finite. Since $P$ is a permutable subgroup of $X$ it follows that $\langle F,P\rangle=FP$ and so $P$ has finite index in $FP$. Consequently $P/P_X$ has finite index in $FP/P_X$, and we  deduce from Lemma \ref{l:folk1} that $FP/P_X$ is residually finite. 
 Let $M/P_X$ be a normal subgroup of finite index in $FP/P_X$. If $M$ is nilpotent, then $FP$ is nilpotent-by-finite and therefore nilpotent by Lemma \ref{l:folk2}. Since $P$ is non-nilpotent we deduce that  $M$ is not nilpotent and so every subgroup of $FP$ containing $M$ is a permutable subgroup of $X$. Thus $FP/M$ is quasihamiltonian (i.e. all of its subgroups are permutable) and it follows that $FP/M$ is metabelian. Thus $FP/P_X$ is residually metabelian and hence metabelian. In particular $FP_X/P_X$ is metabelian so $F$ is soluble of derived length at most $d+2$.  It follows that $X$ is locally (soluble of derived length at most $d+2$) and therefore  soluble of derived length at most $d+2$. This contradiction completes the proof.
\end{proof}



\begin{thebibliography}{10}

\bibitem{aA99}
A.~O. Asar, \emph{Locally nilpotent $p$-groups whose proper subgroups are
  hypercentral or nilpotent-by-{C}hernikov}, J. London Math. Soc. \textbf{61}
  (2000), 412--422.

\bibitem{ADE23}
S.~Atlihan, M.~R. Dixon, and M.~J. Evans, \emph{Locally graded groups with all non-nilpotent subgroups permutable}, J. Algebra \textbf{632} (2023), 62--69.

\bibitem{hS01b}
H.~Smith, \emph{Groups with all non-nilpotent subgroups subnormal}, Topics in
  infinite groups, Quad. Mat., vol.~8, Dept. Math., Seconda Univ. Napoli,
  Caserta, 2001, pp.~309--326.

\bibitem{sS72}
S.~E. Stonehewer, \emph{Permutable subgroups of infinite groups}, Math. Z.
  \textbf{125} (1972), 1--16.

\end{thebibliography}
\end{document}